\documentclass[12pt]{amsproc}
\usepackage[cp1251]{inputenc}
\usepackage[english]{babel}
\usepackage{amssymb}

\numberwithin{equation}{section}
\newtheorem{tm}{Theorem}
\newtheorem*{tm1}{Theorem 1}
\newtheorem{lem}[tm]{Lemma}
\newtheorem{cor}[tm]{Corollary}

\renewcommand{\Re}{\operatorname{Re}}
\renewcommand{\Im}{\operatorname{Im}}

\newcommand{\la}{\lambda}

\newcommand{\ve}{\varepsilon}

\newcommand{\vp}{\varphi}

\begin{document}

\title[]{
Eigenvalue Asymptotics of Perturbed Self-adjoint  Operators }
\author[]{A.~A.~Shkalikov}
\thanks{This work is supported by Russian Foundation for Basic Research (project  No. 10-01-00423-a )}



\maketitle


 \vspace{1cm}

{\bf Abstract}.  We study perturbations of  a self-adjoint
positive operator $T$, provided that a perturbation operator $B$
satisfies "local" \ subordinate condition $\|B\varphi_k\|
\leqslant b\mu_k^{\beta}$ with some $\beta <1$ and $b>0$. Here
$\{\varphi_k\}_{k=1}^\infty$ is an orthonormal system of the
eigenvectors of the operator $T$  corresponding to the eigenvalues
$\{\mu_k\}_{k=1}^\infty$.  We introduce the concept of
$\alpha$-non-condensing sequence and prove the theorem on the
comparison of the eigenvalue-counting functions of the operators
$T$ and $T+B$. Namely, it is shown that if $\{\mu_k\}$ is
$\alpha-$non-condensing then the difference of the
eigenvalue-counting functions is subject to relation
$$
|n(r,\, T)- n(r,\, T+B)| \leqslant C\left[ n(r+ar^\gamma,\, T) -
n(r-ar^\gamma,\, T)\right] +C_1
$$
with some constants $C, C_1, a$ and $\gamma = \max(0,\, \beta,\,
2\beta +\alpha -1)\in [0,1)$.

The results of this paper are published in \cite{Sh2}.

\vspace{1cm}
\section{Introduction}
Throughout this paper, $T$ shall stand for a self-adjoint and
bounded below operator with domain $\mathcal{D}(T)$ acting in a
separable Hilbert space $\mathcal{H}$.  We always suppose that $T$
has a discrete spectrum which is denoted by $\{\mu_k\}_{k=
1}^\infty$ and each eigenvalue $\mu_k$ is repeated in the sequence
in accordance with its geometric multiplicity. A complete
orthonormal system of the eigenvectors that correspond to these
eigenvalues we denote by $\{\varphi_k\}_{k= 1}^\infty.$ For
convenience we always assume that $1<\mu_k\leq \mu_{k+1}$ for all
integers $k\geq 1$. Let
\begin{equation}\label{1}
n(r,\,  T) =\sum_{\mu_k < r} 1
\end{equation}
be the eigenvalue-counting function of the operator $T$. We
suppose that there is a positive number $\alpha$, such that
\begin{equation}\label{2}
\overline\lim_{r\to\infty} \frac{n(r,\, T)}{t^\alpha} =C<\infty.
\end{equation}
This condition is natural as it is fulfilled for a large class of
differential operators (ordinary and with partial derivatives on a
bounded domain in $\mathbb R^n$, see \cite{RSS}, for example). In
the case $\alpha =1$ we say that a sequence
$\{\mu_k\}_{k=1}^\infty$ is {\it non-condensing} if there is a
number $l$ such that each segment $(t,t+1],\ t\in \mathbb R^+,$
contains  at most $l$ eigenvalues of the operator $T$. Obviously,
this condition is equivalent to the following: {\it for all $t>1$
the inequality $n(t+0,\, T) - n(t-1,\, T) \leqslant l$ holds}.  In
the case $\alpha \ne 1$ we introduce  the concept of {\it
$\alpha$-non-condensing sequence}. Namely, we say that a sequence
$\{\mu_k\}_{k=1}^\infty$ satisfying the condition \eqref{2} is
{\it $\alpha$-non-condensing} if a sequence
$\{\mu_k^\alpha\}_{k=1}^\infty$ is non-condensing, or equivalently
$n(t^{1/\alpha}+0,\, T) - n((t-1)^{1/\alpha}, \, T) \leqslant l$
with some $l\in \mathbb N$.

The goal of this paper is to obtain results on the distribution of
the eigenvalues for the perturbations $A= T+ B$, provided that the
perturbation operator $B$ satisfies the conditions
\begin{equation}\label{3}
\mathcal D(B) \supset \mathcal D(T),  \quad \|B\varphi_k\| \leq
b\mu_k^\beta,   \quad -\infty< \beta < 1,
\end{equation}
where $b$ is a constant. The main our result is as follows.

\begin{tm1} Let  conditions \eqref{2}  and \eqref{3} be fulfilled and the
sequence $\{\mu_k\}_{k=1}^\infty$ be  $\alpha$-non-condensing.
Assume that
\begin{equation}\label{gamma}
\gamma =\max(0, \, \beta,\,  2\beta+\alpha -1) <1.
\end{equation}
 Then the spectrum of the operator $A=T+B$ consists of
isolated eigenvalues $\{\lambda_k\}$ and there exist positive
constants $a, C$ and $C_1$ such that the eigenvalue-counting
function
$$
n(r,\, A) = \sum_{|\lambda_k| <r} 1
$$
is subject to the relation
\begin{equation}\label{mfor}
|n(r,\, A) - n(r,\, T)| \leq C S_\gamma(r)+ C_1,
\end{equation}
where
$$
S_\gamma(r) = n(r+ar^{\gamma},\, T) -  n(r-ar^{\gamma},\, T).
$$
\end{tm1}
Viewing in mind  applications, it is worth mentioning the
following corollary.

\begin{cor} Assume that
$$
n(r,T) = r^\alpha +O(r^\eta), \qquad \text{with some}\ \,  \eta
<\alpha.
$$
Then under assumptions of Theorem 1 we have
$$
|n(r,\, A) - n(r,\, T)| = O(r^{\alpha +\gamma -1}) + O(r^\eta)
$$
\end{cor}
The proof is obvious.

 The asymptotic behavior of the eigenvalues of self-adjoint
operators and their perturbations has long been studied by many
authors. Asymptotic formulae for the Sturm-Liouville operator were
  found in 19-th century.  The first general results on the
eigenvalue distribution of the eigenvalues of ordinary
differential operators were obtained by Birkhoff \cite{Bi}, and
for partial differential operators by Weyl \cite{We}. Keldysh
\cite{Ke} proved the first result for relatively compact
perturbations of general self-adjoint operators in Hilbert space
using Tauberian technique. The most complete and sharp results for
compact perturbations and for the so-called $\beta$-subordinate
perturbations of self-adjoint operators  are due to Markus and
Matsaev \cite{MM} (see more details in \cite [Ch.1]{Mar}.
Additional information on eigenvalue distribution of self-adjoint
operators and their perturbations can be found in the book of
Naimark \cite{Na}, the survey of Rosenblum, Solomyak and Shubin
\cite{RSS}, the book of Markus \cite{Mar}, the survey of
Agranovich \cite{Ag}.

Here we remark that the main novelty of our paper is the
subordinate condition \eqref{3}. The  sharpest results on the
comparison of spectra of original and perturbed operators which
are due to Markus and Matsaev \cite{MM},   dealt with subordinate
conditions of the form
\begin{equation}\label{4}
\|Bf\| \leq C\|T^\beta f\| \quad \forall \ f\in \mathcal D(B)
\supset \mathcal D(T^\beta)
\end{equation}
or
\begin{equation}\label{5}
\|Bf\| \leq C\|Tf\|^\beta \|f\|^{1-\beta}\ \ \forall f\in \mathcal
D(T).
\end{equation}

Here $\beta <1$ and $C$ is a constant independent on $f$. We also
note that the second condition here is weaker than the first one.
Obviously, our condition \eqref{3} is essentially weaker than
conditions \eqref{5}. We shall demonstrate this by a simple
example.

Consider the self-adjoint operator $T = i\frac d{dx}$ on the
domain
$$
\mathcal D(T) = \{ f: f\in W^1_2(0, 2\pi), f(0)=f(2\pi)\}
$$
in the Hilbert space $\mathcal H = L_2[0, 2\pi]$. The eigenvalues
of $T$ are equal $\mu_k = k, k\in \mathbb Z$, and the
eigenfunctions coincide with trigonometric system
$\varphi_k=\{e^{ikx}\}_{k=-\infty}^\infty.$  Then, consider as a
perturbation the multiplication operator $Bf = b(x)f(x)$, where
$b(x)\in L_2(0, 2\pi)$ and has sufficiently strong singularity at
some point $x_0\in [0,2\pi]$, say  $b(x) = \frac
1{\ln(x/4\pi)\sqrt x}$. Then $\|B\varphi_k\| \leqslant \|b(x)\|$,
i.e. the condition \eqref{3} holds with $\beta =0$.  It is known
\cite{Ev} that $\mathcal D(|T|^\beta) =W^\beta_2$ for any $0\leq
\beta < 1/2$, where $W^\beta_2 =W^\beta_2(0,2\pi)$ is the Sobolev
space with smooth index $\beta$. It can easily be verified that
the function $f_0(x)= \ln(x/4\pi)$ belongs to the space
$W^\beta_2$ for any $\beta < 1/2$. Hence, for these values of
$\beta$ we have $Bf_0\ne L_2(0,\, 2\pi)$, while $|T|^\beta f_0 \in
L_2(0,\, 2\pi)$. Therefore, condition \eqref{4} is not fulfilled
for any $\beta <1/2$.  Then, the same is true for condition
\eqref{5}, because the validity of \eqref{5} with some $\beta<1$
implies the validity of \eqref{4} with any $\beta' <\beta$ (see
\cite[Ch 1] {Mar}, for example). This example shows that in
particular situations the subordinate condition \eqref{3} can be
much more effective than \eqref{4} or \eqref{5}.  Simultaneously
we have to say that Theorem 1 does not generalize the Markus
-Matsaev theorem \cite{MM}. Condition \eqref{4} or \eqref{5}
implies the validity of Theorem 1 with the function
$$
S(r) = n(r+ar^\beta,\, T) -  n(r-ar^\beta,\, T),
$$
i.e. $\gamma$ can be replaced by $\beta$. Therefore,
  assuming a
weaker assumption \eqref{3} instead of \eqref{4} or \eqref{5}, we
get the same estimate as in the Markus-Matsaev theorem only in the
case $\beta \leqslant 2\beta+\alpha -1$, i.e.  $\beta +\alpha
\leqslant 1$.  Otherwise, we have to pay for a weaker assumption
getting  estimate \eqref{mfor} which is less sharp.


Finally, we remark that "local" subordinate condition \eqref{3}
was originated in author's paper \cite{Sh0}, where Theorem 1 was
proved  for the case  $\alpha =1$ and $\beta =0$.

\section {Proof of Theorem 1}

First, we shall prove Theorem 1 for the case $\alpha =1$. In this
case the proof is more transparent and technically much easier.
Later on we will explain how to overcome  the technicalities in
the case $\alpha \ne 1$.  While proving this result we will use
the trick of "artificial lacuna" proposed by Markus and Matsaev in
\cite{MM}. However, the implementation of this trick will be
organized in a technically different way. Our plan to prove the
theorem is the following.  First, we prove relation (1.4) for a
fixed $r$, provided that the interval $(r-2ar^{2\beta},
r+2ar^{2\beta})$ does not contain the eigenvalues of the operator
$T$,  where $a$ is a certain number depending on the numbers $b$
and $\beta$ and independent of $r$. For such $r$ we show the
equality $n(r,A) = n(r,T)$. Then, for each fixed $r$ we construct
a finite rank  self-adjoint operator   $K_r$ commuting with $T$
such that
\begin{itemize}
\item[(i)] the operator $T_r = T-K_r$ has no eigenvalues in the interval
$(r-ar^{2\beta}, r+ar^{2\beta})$;

\item[(ii)] the property \eqref{3} remains valid for $T-K_r$ with the
constant   $ 2b$ instead of $b$;

\item[(iii)] The inequality $|n(r, T)- n(r, T_r)|\leqslant S_\gamma(r)$
holds.
\end{itemize}

   Then we apply the Weinstein-Aronszain formula
from the theory of perturbation determinants (see \cite[Ch.5]{Ka})
\begin{equation}\label{WA}
n(r,A) = n(r, T-K_r +B) +\nu(r,h),
\end{equation}
where  $\nu(r,h)$   denotes the difference between the numbers of
zeros and poles of the scalar meromorphic function
$D(\la):=\det(1-K_r(\la -T+K_r-B)^{-1})$, that lie in the
rectangle $\mathcal R$ which is bounded by vertical lines
$\Re\lambda =r, \ \Re\lambda =-R$ and the horizontal  lines
$\Im\lambda = \pm R$ with sufficiently large $R=R(r)$.
 We remark that formula \eqref{WA}
 can  easily be proved by using the identity
$$
1-K_r(\la-T+K_r -B)^{-1} =(\la -A)(\la
-T+K_r-B)^{-1}.
$$
Since the operator $T_r = T-K_r$ has no eigenvalues in the
interval $(r-2ar^{2\beta}, r+2ar^{2\beta})$ we get $n(r, T_r +
B)=n(r, T_r)$. On the other hand, by construction we have $|n(r,
T_r) - n(r, T)| \leqslant N=S_\gamma(r)$.  Therefore, we will
prove \eqref{mfor} if we show that the function $|\nu(r,h)|$ is
bounded by the right hand-side of \eqref{mfor}. To show the latter
assertion is the main technical difficulty in the proof of Theorem
1 which we shall divide into several steps.

{\it Step 1.} We will use in the sequel  the following result from
complex analysis.

\begin{lem}\label{lem1}

Let $f$ be a function that is bounded and analytic in the
rectangle
\begin{equation}\label{Pi}
\Pi\ = \{\la: \ \, |\Re\la -r| <c, \ \, |\Im\la| <d\}.
\end{equation}
For  $\delta \in (0,1)$, set  $c' =c(1-\delta)\ \,
and\  d'=d(1-\delta)$, and denote by  $\Pi'$ the rectangle defined
by \eqref {Pi} where  $c$ and $d$ are replaced  by $c'$  and $d'$.
Denote
$$
 M= \sup_{\la\in \Pi} f(\la),\quad  M' = \sup_{\la\in \Pi'} f(\la).
 $$
Then there is a constant $C$ depending on $\delta$ and the ration
$c/d$ and independent of~$f$ such that the following holds:
\begin{itemize}
\item[(i)] The number $n_f(\Pi')$ of zeros of the function~$f$ inside the rectangle
$\Pi'$ is subject to the estimate
\begin{equation}\label{n_f}
n_f(\Pi') \leqslant C(\ln M-\ln M').
\end{equation}

\item[(ii)]
If $\gamma$ is a straight  line segment contained  in  $\Pi'$ that
does not pass through the zeros of $f$ in $\Pi'$, then the
variation of the argument of the function $f$ along $\gamma$ is
subject to the same estimate
\begin{equation}\label{arg}
|\, [\arg f(\la)]_\gamma | \leqslant C(\ln M-\ln M').
\end{equation}
\end{itemize}


\end{lem}

\begin{proof} Some versions of this assertion can be found in the
monograph  \cite[Ch.~1]{Le} and in \cite{Sh1}.  In the form
presented here  this result  is contained in \cite[Lemmas 1.1 and
1.3]{MM}.
\end{proof}

{\it Step 2.}

\begin{lem}\label{lem2}
Under assumptions of Theorem 1 the operator $A=T+B$ has discrete
spectrum.
\end{lem}
\begin{proof}
Since $T$ is self-adjoint the following representation for the
resolvent holds
\begin{equation}\label{res}
(\lambda - T)^{-1} = \sum_{k=1}^\infty \frac{(\cdot ,
\vp_k)\vp_k}{\lambda - \mu_k}.
\end{equation}
Without loss of generality we have assumed that the point $\lambda
=0$ belongs to the resolvent set of $T$. Denoting $f_k = (f,
\vp_k)$  and taking into account that $\sum |f_k|^2 = \|f\|^2$ we
get
\begin{multline}\label{comp}
\left\|BT^{-1}f - \sum_{k=1}^N \frac{f_kB\vp_k}{\mu_k}\right\|=
\left\| \sum_{k=N+1}^\infty \frac{f_kB\vp_k}{\mu_k}\right\|
\leqslant
 \|f\|^2 \sum_{k=N+1}^\infty
\frac{\|B\vp_k\|^2}{\mu_k^2}\\ \leqslant b^2\|f\|^2
\sum_{k=N+1}^\infty \mu_k^{2\beta -2} \leqslant \ve \|f\|^2,
\end{multline}
where $\ve \to 0$  as $N\to\infty$. The latter assertion holds
since condition \eqref{2} with $\alpha =1$ implies $\mu_k \geq
C^{-1}k$, hence, the series $\sum \mu_k^{2\beta -2} \leqslant
C^{2-2\beta} \sum k^{2\beta -2}$  converges (here we use our
assumption \eqref{gamma} which implies $\beta <1/2$, provided that
$\alpha =1$). The estimate \eqref{comp} shows that $BT^{-1}$ is
compact, therefore $B$ is a relatively compact perturbation of
$T$. Then the discreteness of the spectrum of $T+B$ follows from
lemma of Keldysh (see, for example, \cite[Lemma 3.6]{Mar}).
\end{proof}

{\it Step 3.}
\begin{lem}\label{lem3}
Let   the spectrum $\{\mu_k\}_1^\infty$ of the operator $T$ form a
non-condensing sequence.  Equivalently, there is a number $l\in
\mathbb N$ such that
\begin{equation}\label{l}
n(t+0) - n(t-1) \leqslant l  \qquad \text{for all} \ \ t\geqslant
 1.
\end{equation}
Then there is a continuous piece-wise linear function $\psi(t)$
such that
$$
|n(t) - \psi (t)| \leqslant l
$$
and
$$
|\psi'(t)| \leqslant l.
$$
\end{lem}

\begin{proof}
Without loss of generality  we have already assumed  that
$\mu_1>1$. Define the integers $s_m:=n(m+0)$. Then the segments
$\Delta_m=(m-1,\, m], \ m=1,2,\dots,$ contain $l_m=s_m
-s_{m-1}\leqslant l$ eigenvalues from the sequence
$\{\mu_k\}_{k=1}^\infty$.
 Now define
the function $\psi(t)$ on the interval $\Delta_{m+1} =(m,\, m+1]$
as follows
$$
\psi(t) = s_m +l_m (t-m), \qquad t\in (m,\, m+1].
$$
It follows from the construction that $|\psi(t)- n(t)| \leqslant
\sup\{l_m\} =l$ and $\psi'(t) \leqslant l$. The lemma is proved.
\end{proof}

{\it Step 4.} Let us prove the following lemma.
\begin{lem}\label{lem4}
 Let $a$  be a fixed positive number and suppose that  the interval
  $(r-2ar^{2\beta}, r+2ar^{2\beta})$  does not contain  the points  $\mu_k$
  of the spectrum of the operator $T$. Assume also that the
  constant $b$ participating in condition \eqref{3} is  such that
\begin{equation}\label{b<}
a\geqslant 48\, l\, b^2.
\end{equation}
Then  the following estimate
is valid  in the strip $r-ar^{2\beta}\leqslant \Re\la\leqslant
r+ar^{2\beta}:$
\begin{equation}\label{1/2}
\sum_{k=1}^\infty \frac{\|B\vp_k\|^2}{|\la-\mu_k|^2}<\frac 14,
\end{equation}
provided that $r\geqslant C$ where $C=C(a,\beta)$ depends only on
$a$ and $\beta$ (hence, only on $l, b$ and $\beta$).
\end{lem}

\begin{proof}
Denote $\lambda = \Re\lambda +i \Im\lambda: =\sigma +i\tau$ and
$r^- = r-2ar^{2\beta}$,  $r^+ = r+2ar^{2\beta}$.  We have to
estimate the sum for the values $\sigma\in(r-r^{2\beta}, r+
r^{2\beta})$ and  $\tau \in \mathbb R$.  Recall that according to
Lemma~\ref{lem3} we have the representation
$$
n(t) =\psi(t) +\zeta(t), \qquad \text{where}  \ \,
|\psi'(t)|\leqslant l, \ \, |\zeta(t)| \leqslant l,   \ \, n(t)
=n(t,\, T).
$$
 Using condition \eqref{3} we obtain
\begin{multline}\label{first}
\sum_{k=1}^\infty\frac{\|B\vp_k\|^2}{| \la-\mu_k|^2} \leqslant
 b^2\sum_{k=1}^\infty\frac{\mu_k^{2\beta}}{(\sigma-\mu_k)^2
 +\tau^2} =
\\  b^2\int_1^\infty \frac{t^{2\beta}\, d n(t)}{(\sigma-t)^2
 +\tau^2}\leqslant
b^2l \left( \int_1^{r^-} +\int_{r^+}^\infty\right)
\frac{t^{2\beta}\, dt}{(\sigma -t)^2 +\tau^2} + \\
b^2 \left( \int_1^{r^-} + \int_{r^+}^\infty\right)
\frac{\left(2\beta t^{2\beta - 1} [(\sigma -t)^2 +\tau^2]
+2t^{2\beta}|\sigma-t|\right)|\zeta(t)|\, dt}{[(\sigma -t)^2
+\tau^2]^2}.
\end{multline}
Remark that we have integrated by parts while transforming here
the integrals. Taking into account the inequalities
$$
|\sigma -t|/ \sqrt{(\sigma-t)^2 +\tau^2} \leqslant 1, \quad 2\beta
<1, \quad |\zeta(\xi)| \leqslant l,
$$
we estimate the last sum of the integrals as follows
\begin{multline}\label{sec}
\leqslant l \left( \int_1^{r^-} + \int_{r^+}^\infty\right)
\left[\frac{2\beta t^{2\beta - 1}}{(\sigma -t)^2 +\tau^2} +
\frac{2t^{2\beta}}{\left((\sigma -t)^2
+\tau\right)^{3/2}}\right]\, dt \leqslant \\ 3l\left( \int_1^{r^-}
+ \int_{r^+}^\infty\right) \frac{ t^{2\beta}\, dt}{(\sigma -t)^2
+\tau^2} \leqslant  3l\left( \int_1^{r^-} +
\int_{r^+}^\infty\right) \frac{ t^{2\beta}\, dt}{(\sigma -t)^2},
\end{multline}
provided that $\min [(\sigma-r^-), (r^+-\sigma)] =
ar^{2\beta}\geqslant 1.$ Therefore, to prove lemma it is
sufficient to estimate the integral
\begin{multline}\label{third}
\left( \int_1^{r^-} + \int_{r^+}^\infty\right) \frac{ t^{2\beta}\,
dt}{(\sigma -t)^2} \leqslant \int_{\sigma -r^-}^\sigma
\frac{(\sigma-\xi)^{2\beta}}{\xi^2}\, d\xi + \int_{r^+
-\sigma}^\infty\frac{(\xi+\sigma)^{2\beta}}{\xi^2}\, d\xi\leqslant
\\
 \sigma^{2\beta} \int_{ar^{2\beta}}^\infty\frac 1{\xi^2} \, d\xi +
\int_{ar^{2\beta}}^\sigma \frac {(2\sigma)^{2\beta}}{\xi^2} \,
d\xi +\int_\sigma^\infty\frac {(2\xi)^{2\beta}}{\xi^2} \, d\xi <
\frac{(1+2^{2\beta}) \sigma^{2\beta}}{ar^{2\beta}}
+\frac{2^{2\beta}}{(1-2^{2\beta})}\, \sigma^{2\beta -1} <\\
\frac{1+2^{2\beta}}a\, \left(1+ar^{2\beta -1}\right)^{2\beta}
+\frac{2^{2\beta}}{(1-2\beta)}\,
\left(r-ar^{2\beta}\right)^{2\beta-1}< \frac 3a
-\frac{\left(2-2^{2\beta}\right)}a +Cr^{2\beta-1} < \frac 3a,
\end{multline}
provided that $r$ is sufficiently  large, i.e. $r^{1-2\beta}
\geqslant Ca/(2-2^{2\beta})$ where $C$ depends only on $\beta$ and
$a$. Now, the assertion of lemma  straightly  follows from
\eqref{first}, \eqref{sec} and \eqref{third}.  The Lemma is
proved.
\end{proof}

{\it Step 5.} Here we estimate the left hand-side of \eqref{1/2}
outside the parabolic domain defined as follows
\begin{equation}\label{Par}
P(h, 2\beta) = \{\la:  |Im\lambda|\leqslant h(Re\lambda
)^{2\beta}, Re\lambda \geqslant 0\}
\end{equation}

\begin{lem}\label{lem5}
Let the conditions of Theorem 1 hold.
Let $h$ be a positive number and
$$
\sigma_h = \left[ \frac{2h}{\pi(1-2\beta)(2^{1-2\beta}
-1)}\right]^{1/(1-2\beta)}.
$$
Then for all $\lambda =\sigma+i\tau$ lying in the half-plane
$\sigma \geqslant \sigma_h$     and  outside the parabola $P(h,\,
2\beta)$ defined by \eqref{Par} the following estimate holds
\begin{equation}\label{out}
\sum_{k=1}^\infty \frac{\|B\vp_k\|^2}{|\la-\mu_k|^2}<\frac{6\pi
b^2 l}{h}.
\end{equation}
\end{lem}

\begin{proof}
Repeating the arguments in proving Lemma \ref{lem4} (see estimates
\eqref{first} and \eqref{sec}) we get
\begin{equation} \label{gen}
\sum_{k=1}^\infty \frac{\|B\vp_k\|^2}{|\la-\mu_k|^2}< 4b^2 l
\int_1^\infty \frac{t^{2\beta}\, dt}{(\sigma -t)^2 + \tau^2}.
\end{equation}
Further we proceed
\begin{multline*}
\int_1^\infty \frac{t^{2\beta}\, dt}{(\sigma -t)^2 + \tau^2} <
\int_0^\sigma \frac{(\sigma -\xi)^{2\beta}}{\xi^2+\tau^2}\, d\xi +
\left(\int_0^\sigma +\int_\sigma^\infty\right) \frac{(\sigma
+\xi)^{2\beta}}{\xi^2+\tau^2}\, d\xi<
\\
\sigma^{2\beta}\frac\pi{2\tau} +(2\sigma)^{2\beta}\frac\pi{2\tau}
+\int_\sigma^\infty\frac{(2\xi)^{2\beta}}{\xi^2} \, d\xi.
\end{multline*}
Outside the parabola $P_h$ we have $\tau > h \sigma^{2\beta}$.
Therefore, for $\lambda= \sigma+i\tau \notin P_h$ the right
hand-side of the last inequality can be estimated as follows
$$
< \frac{3\pi}{2h} -\frac{(2-2^{2\beta})\pi}{2h}
+\frac{2^{2\beta}}{(1-2\beta)}\ \sigma^{2\beta-1} \leqslant
\frac{3\pi}{2h},
$$
provided that $\sigma \geqslant \sigma_h$. This ends the proof of
lemma.
\end{proof}

{\it Step 6.}

\begin{lem}\label{lem6}
Let conditions of Lemma \ref{lem4} hold. Then there is a large
number $R=R(r)$ such that the estimate \eqref{1/2} holds on the
boundary of the rectangle $\mathcal R_r$ whose horisontal sides
are the segments of the lines $\Im\lambda =\pm R$ and the
vertical sides are the segments of the lines $\Re\lambda =-R$ and
$\Re\lambda =r$.
\end{lem}

\begin{proof}
The validity of estimate \eqref{1/2} on the line $\Re\lambda =r$
is proved in  Lemma \ref{lem4}. It follows from the proof of
 Lemma \ref{lem5} that the left hand-side of \eqref{1/2} obeys the estimate
$\leqslant C(\Im\lambda)^{-1}$ as $\Im\lambda\to\infty$ uniformly
in the half-plane $\Re\lambda <\sigma_h$. We do not show here
details because they are obviously seen from the proof of  Lemma
\ref{lem5}. Finally, the estimate on the line $\Re\lambda = -R$
with sufficiently large $R$ also follows easily from the above
representations.
\end{proof}

{\it Step 7.} Our goal at this step is to show the equality $n(r,
T_r + B) = n(r, T_r)$, where $T_r$ is a special finite-dimensional
"correction" of the unperturbed operator $T$. First we construct
the operator $T_r$.

Take a positive $a$ such that the inequality
\begin{equation}\label{2b<}
a\geqslant 96\,  b^2\, l
\end{equation}
holds. We pay attention  that this condition differs from
\eqref{b<} by changing $l$ to $2l$ (we will use this further). Fix
a number $r$ such that $r-2ar^{2\beta} >1.$  Define the interval
$\Delta_r = (r-2ar^{2\beta},\,  r+2ar^{2\beta})$ and the operator
$$
K_r = 4ar^{2\beta}  \sum_{\mu_k\in\Delta_r} (\cdot,\,
\varphi_k)\varphi_k.
$$
Obviously, $K_r$ is a self-adjoint operator of finite rank not
exceeding the value
\begin{equation}\label{N}
N= n(r+2ar^{2\beta},\, T) -  n(r-2ar^{2\beta},\, T)
\end{equation}
Now define $T_r= T+K_r$.  Obviously, this operator preserves the
system of eigenfunctions
 $\{\vp_k\}_1^\infty$ but changes the eigenvalues lying in the interval
 $\Delta_r=(r-2ar{2\beta},\, r+2ar^{2\beta})$; it shifts them by
$4ar^{2\beta}$ to the  right from this interval. The condition
$T_r\geqslant 1$  is preserved, since $K_r$ is non-negative. The
sequence of the eigenvalues $\{\mu'_k\}_{k=1}^\infty$ of the
operator $T_r$ remains non-condensing  but we have to take into
account that the number $l = \sup_{t>0} \left(\sum_{\mu_k\in[t,
t+1)} 1\right)$ is changed to 2l. Then, by construction $\mu'_k
\geqslant \mu_k$, therefore condition \eqref{3} is preserved for
$T_r$.

Let us estimate the norm of the operator function $B(\lambda
-T_r)^{-1}$ in the strip $\Re\lambda\in (r-ar^{2\beta},\,
r+ar^{2\beta}).$  We apply the method used by Adduci and Mityagin
\cite[\S 4]{AM}  or \cite{AM1}. Let  $f\in \mathcal H, \|f\|=1 $,
and $f_k=(f,\vp_k)$ be the Fuorier  coefficients of the element
$f$. Then
\begin{equation}\label{AM}
\|B(\la- T_r)^{-1}f\|^2\  =\ \left\|\sum_{k=1}^\infty \ \frac {f_k
B\vp_k}{\la-\mu'_k}\right\|^2\ \leqslant\ \sum_{k=1}^\infty \
|f_k|^2 \sum_{k=1}^\infty\ \frac{\|B\vp_k\|^2}{|\la-\mu'_k|^2}\ =
\ \sum_{k=1}^\infty \ \frac{\|B\vp_k\|^2}{|\la-\mu'_k|^2}.
\end{equation}
Applying  Lemma \ref{lem4} and taking into account that the number
$a$ is selected by \eqref{2b<} instead of \eqref{b<} (because the
number $l$ for $T_r$ has to be changed by $2l$), we obtain  the
following estimate
\begin{equation}\label{1/4}
\|B(\la-T_r)^{-1}\| <\sqrt\frac 14 =\frac 12, \quad \ \ \Re\la \in
(r-ar^{2\beta}, r+ar^{2\beta}).
\end{equation}
Now, by virtue of  Lemma \ref{lem6} take a number $ R = R(r)$ such
that estimate \eqref{1/4} holds on the boundary of the rectangle
$\mathcal R$. Then, for all $t\in [0,1]$ the Riesz projectors
$$
Q_t =\frac 1{2\pi i}\int_{\partial\Omega} (\la- T_r-tB)^{-1}\,
d\la \ = \ \frac 1{2\pi i}\int_{\partial\Omega} (\la- T_r)^{-1}
(1-tB(\la-T_r)^{-1}\, d\la
$$
are well defined and depend continuously on $t$ in the norm
operator topology. By virtue of Sz\" okefalvi-Nagy's lemma (see
\cite[Ch. 1, Lemma 3.1]{GK}) $\dim Q_t =\dim
 Q_\xi$, provided that $\|Q_t- Q_\xi\|<1$. Therefore, it follows from
 the continuity of $Q_t$ that
 \begin{equation}\label{nT}
n(r, T_r) = \dim P_0 = \dim P_1 = n(r, T_r +B).
\end{equation}
This proves lemma.

 {\it  Step 8.} Consider the scalar function
$$
D(\la) = \text{det}(1-K_r(\la- T_r-B)^{-1}).
$$

By virtue of Lemma \ref{lem2}  the spectrum of the operator
$T_r+B$ is discrete. Hence, the operator function  $K(\la): =
K_r(\la-T_r-B)^{-1}$ is meromorphic and its values  are finite
rank operators  for  $\la\notin\sigma (T_r+B)$. Therefore, the
determinant $D(\la)$  is well defined ( as a meromorphic function)
and is equal to the product $\prod_j (1-\la_j(K))$, where
$\la_j(K)$ are the eigenvalues of the operator $K(\la)$. Since
$\dim K_r\leqslant 4ap$,  this product contains at most
$4ar^{2\beta}$  factors.

\begin{lem}\label{lem7}
In the strip $\Re\la\in(r-a r^{2\beta},\, r+ar^{2\beta})$ the
function $D(\la)$ is holomorphic  and is estimated as
\begin{equation}\label{D+}
|D(\la)|\leqslant 9^{N},
\end{equation}
where the number $N$ is defined by \eqref{N}. At the point
$\lambda = r+ihr^{2\beta}$ the following lower estimate holds
\begin{equation}\label{D-}
|D(\la)|\geqslant \left(\frac 12\right)^{N}, \quad \text{provided
that} \ \, h\geqslant 16\, a.
\end{equation}
\end{lem}
\begin{proof}
We shall use the identity
$$
(\la -T_r- B)^{-1} = (\la -T_r)^{-1}(1- B(\la-
T_r)^{-1})^{-1}
$$
and the estimates
$$
\|(\la- T_r)^{-1}\|\leqslant \frac 1{\text{dist} (\la, \sigma
(T_r))}\leqslant \frac 1{ar^{2\beta}}, \quad \|(1- B(\la -
T_r)^{-1})^{-1}\| \leqslant 2
$$
which hold for $\lambda$ in the strip $\Re\la\in (r-ar^{2\beta},\,
r+ar^{2\beta})$. The first estimate here  is valid because $T_r$
is self-adjoint, and the second one is proved in Lemma \ref{lem6}.
In particular, we find that  the operator function $K(\la)$ is
holomorphic in the strip $|\Re\la -r| < ar^{2\beta}$ and its
eigenvalues obey the inequalities
 $$
 |\la_j(K)|\leqslant \|K_r\|\  \|(\la- T_r- B)^{-1}\|
 \leqslant 4ar^{2\beta} \cdot \frac 1{ar^{2\beta}} \cdot 2 = 8.
 $$
 The number of the eigenvalues is equal to the rank of the operator $K(\la),$
 which does not exceed the rank of the operator $K_r$  equal to $N$.
Therefore, the product of $N$ factors of the form $(1-\la_j(K))$
is subject to estimate \eqref{D+}.

Next, by virtue of Lemma \ref{lem4} the estimate
$\|B(\la-T_r)^{-1}\|<1/2$ is valid for $\lambda$ on the line
$\Re\lambda =r$. Using the resolvent estimate for the self-adjoint
operator $T_r$ we get for $\lambda= r+ihr^{2\beta}$
\begin{equation}\label{la_j}
|\la_j(K(\lambda))|<\|K_r\| \ \|(\la- T_r)^{-1}\|\
\|(1-B(\la-T_r)^{-1})^{-1}\| \leqslant 4ar^{2\beta}\cdot \frac
1{r^{2\beta}\, \sqrt{1+h^2}}\, \cdot 2 \, < \frac{8a}h <1/2,
\end{equation}
provided that $h\geqslant 16 a$.  Then $|1-\lambda_j(K)|>1/2$.
Therefore, the product of $N$ factors of this form can be
estimated as follows
$$
|D(\la)|\geqslant \left(1-\frac 12\right)^{N} \geqslant
\left(\frac 12\right)^{N}, \qquad |\Im\la|\geqslant
hr^{2\beta}\geqslant 16\, a\, r^{2\beta}.
$$
The lemma is proved.
\end{proof}

{\it Step 9.} Now we are ready to prove the main lemma.
\begin{lem}\label{lem9}
Fix   numbers $a\geqslant 48\, b^2\, l$ and $h\geqslant 16 a$.
Take sufficiently large $r$ and consider a rectangle $\mathcal R$
defined in Lemma \ref{lem4} with $R> 2hr^{2\beta}$.  Suppose that
the line $\Re\lambda =r$ does not contain the eigenvalues of the
operator $A$.    Then the variation of the argument along the
boundary of the rectangle $\mathcal R$  is subject to the estimate
$$
|[\arg D(\la)]|_{\partial\mathcal R}| \leqslant CN +C_1,
$$
where $N$ is defined by \eqref{N} and $C, \ C_1$  are constants
depending only on $l$ and $b$.
\end{lem}
\begin{proof} We have proved in the previous lemma estimate
\eqref{la_j} for all $\lambda \in \partial\mathcal R$ except the
segment on the line $\Re\lambda =r$ with the endpoints $r\pm
ihr^{2\beta}$. Hence, the variation of the argument of the
functions $(1-\lambda_j(K(\lambda))$ when $\lambda$ varies along
$\partial\mathcal R$ between these points outside this segment
does not exceed $\pi/3$. Then the variation  of the argument of
the function $D$ along this curve is $\leqslant \pi N/3.$

To complete the proof we have to estimate the variation of the
argument along the segment  $I_r=[r - ihr^{2\beta},\, r+
ihr^{2\beta}].$ For this purpose we shall use Lemma \ref{lem1}.
First, we chose a number $a$ satisfying condition \eqref{2b<}.
Then, we take a number $h$, say,   $h= 16\, a\, b^2$ such that
inequality \eqref{D-} holds. Assume that $r\geqslant C$ where $C$
is the constant from Lemma \ref{lem4}. Consider  the rectangle
$R_{a,h}$ bounded by the straight lines $\Re\la = r\pm a, \Im\la
=\pm 2h$ and denote by  $R'_{a,h}$ the twice contracted rectangle
with the same center at the point $r$. Lemma \ref{lem1} together
with estimates  \eqref{D+} and \eqref{D-} imply that the variation
of the argument of the function  $D$  along the segment $I_r$
(provided that this segment does not passes through the zeros of
the function $D$) does not exceed $C'(\ln 9 + \ln 2)N$ where $C'$
is an absolute constant. This proves the lemma.

\end{proof}

{\it Step 10.} It follows from Lemma \ref{lem9} that the
difference between the number of zeros and poles of the function
$D$  inside the  rectangle $\mathcal R=\mathcal R_r$ does not
exceed $C N \leqslant C'(b^2 l)N$ where $C'$  is an absolute
constant. Note also that by construction of $T_r$ we have
$0\leqslant n(r, T) -n(r, T_r) \leqslant N$. Therefore,  formula
\eqref{WA} gives $|n(r, A) -n(r, T)| \leqslant C N$, provided that
$r\geqslant C_0$. Taking $C_1 = n(C_0,T)$ we get the assertion of
Theorem~1. We have only to explain what to do with exceptional
values of $r$ when the segment $I_r$ passes through the zeros of
the function $D$ which coincide with the eigenvalues of $A$. To
explain this we remark that all these zeros of the function $D$
lie in the rectangle $R'(a,h)$ and by virtue of Lemma \ref{lem1}
the number of these zeros is $\leqslant CN$. Therefore, the jump
of the function $n(r, A)$  does not exceed this value, and the
relation \eqref{mfor} remains valid for all $r\in \mathbb R^+$.
This ends the proof of Theorem~1 for the case $\alpha = 1$.

{\it Step 11.} Let $\alpha >0$ and $\gamma : = 2\beta +\alpha -1$,
$0\leqslant\gamma <1$.  An important step in the proof of the
theorem for the case $\alpha =1$ was made in Lemma \ref{lem4}. In
the general case we also have the estimate
\begin{equation}\label{B}
\sum_{k=1}^\infty\frac{\|B\vp_k\|^2}{| \la-\mu_k|^2} \leqslant
 b^2\sum_{k=1}^\infty\frac{\mu_k^{2\beta}}{(\sigma-\mu_k)^2
 +\tau^2} =
  b^2\int_1^\infty \frac{t^{2\beta}\, d n(t)}{(\sigma-t)^2
 +\tau^2}.
\end{equation}

Since the sequence $\{\mu_k\}_{k=1}^\infty$ is
$\alpha$-non-condensing, the function  $n(\xi^{1/\alpha})=
n(\xi^{1/\alpha},\, T)$ by virtue of Lemma \ref{lem3} can be
represented in the form
$$
n(\xi^{1/\alpha}) =\psi(\xi) +\zeta(\xi), \quad 0\leqslant
\psi'(\xi) \leqslant l, \ \, |\zeta(\xi)| \leqslant l.
$$
Denote  $\lambda =\sigma +i\tau$, $r^- = r- 2ar^\gamma,\ r^+ =
r+2ar^\gamma$ and assume that this interval does not contain  the
eigenvalues of the operator $T$.  Then, we can rewrite the
integral in the right hand-side of \eqref{B} as follows
\begin{multline}\label{est1}
\int_1^\infty \frac{\xi^{2\beta/\alpha} \,
d[\psi(\xi)+\zeta(\xi)]}{(\sigma-\xi^{1/\alpha})^2 +\tau^2}
\leqslant \\
l\int_1^\infty \frac{\xi^{2\beta/\alpha} \,
d\xi}{\left(\sigma-\xi^{1/\alpha}\right)^2 +\tau^2} +
\int_1^\infty\left|\, \left[ \frac{\xi^{2\beta/\alpha}}
{\left(\sigma-\xi^{1/\alpha}\right)^2 +\tau^2}\right]'\, \right|\
|\zeta(\xi)|\, d\xi.
\end{multline}

The second integral in the right hand-side of \eqref{est1} obeys
the estimate
\begin{equation}\label{est2}
\leqslant l\int_1^\infty \left[\frac{2\beta t^{2\beta -1}}
{(\sigma -t)^2 +\tau^2} + \frac{2 t^{2\beta}} {\left[(\sigma -t)^2
+\tau^2\right]^{3/2}}\right]\, dt.
\end{equation}
 Recalling that $\mu_k \notin (r^-, r^+)$ we can replace the
 integral $\int_1^\infty$ by $\int_1^{r^-} +\int_{r^+}^\infty$.
 Then, the left hand-side in \eqref{est1} is subject to estimate
 (for all $\tau \in \mathbb R$)
 \begin{equation}\label{est3}
 \leqslant \alpha l\left(\int_1^{r^-}
 +\int_{r^+}^\infty\right)\left[
 \frac{ t^\gamma}{(\sigma - t)^2 +\tau^2}  +\frac{ 2\beta t^{2\beta
 -1}}{(\sigma - t)^2 +\tau^2}
 + \frac{ 2t^{2\beta}}{(\sigma - t)^3}\right]\, dt
 \end{equation}
Here the integral from the first summand  can be estimated in the
same way as in Lemma \ref{lem4}, namely, it is $\leqslant
\frac{(1+2^\gamma)}a +C r^{\gamma -1}$. There are no problems with
the estimation of the integral from the second summand because
$2\beta -1 < \gamma$. Finally, let us estimate, for example, the
first integral from the third summand. We have
$$
\int_1^{r^-} \frac{t^{2\beta}}{(\sigma-t)^3}\, dt \leqslant
\sigma^{2\beta} \int_{ar^\gamma}^{\sigma -1} \frac 1{x^3}\, dx
\leqslant \frac{\sigma^{2\beta}}{a^2 r^{2\gamma}} \leqslant \frac
2{a^2},
$$
provided that $r>C=C(a,\beta)$ and $\beta \leqslant \gamma$. In
the case $\beta > \gamma = 2\beta +\alpha -1$ we have to put
$\gamma =\beta$. Then the last estimate holds. The previous
estimates  for the first and the second summand in \eqref{est3}
are also valid because $2\beta +\alpha -1 <\beta$.

Therefore, we have proved that the assertion of Lemma \ref{lem4}
remains valid in the general case $\alpha >0$ if the number
$2\beta$ is replaced by $\gamma =\max (0,\, \beta,\, 2\beta+\alpha
-1)$. All the other arguments in proving Theorem 1 for the general
case remain the same with obvious changes.

\bigskip
\bigskip
Department of Mechanics and Mathematics,

Moscow Lomonosov State University

\medskip

 e-mail: ashkalikov$@$yahoo.com

\end{document}